\documentclass[11pt]{amsart}

\usepackage{amsfonts,amsmath}
\usepackage{amssymb, latexsym}
\usepackage{amscd,amsthm}
\usepackage[all]{xy}

\newcommand{\marginlabel}[1]%
  {\mbox{}\marginpar{\raggedleft\hspace{0pt}\bfseries\sf#1}}

\def\NN{{\mathbb N}}
\def\CC{{\mathbb C}}
\def\PP{{\mathbf P}}
\def\OO{{\mathcal O}}

\def\F{\mathcal{F}}
\def\I{\mathcal{I}}
\def\E{\mathcal E}

\theoremstyle{plain}
\newtheorem*{introdeff}{Definition}
\newtheorem{thmalpha}{Theorem}
\newtheorem{coralpha}[thmalpha]{Corollary}

\newtheorem{thm}{Theorem}[section]
\newtheorem{prop}[thm]{Proposition}
\newtheorem{cor}[thm]{Corollary}
\newtheorem{lem}[thm]{Lemma}

\theoremstyle{definition}
\newtheorem{deff}[thm]{Definition}
\newtheorem{rem}[thm]{Remark}
\newtheorem{term/not}[thm]{Terminology/Notation}
\newtheorem{hyp}[thm]{Hypothesis}

\newcommand\coker{\operatorname{coker}}
\newcommand\Pico{\operatorname{Pic^0}}
\newcommand\codim{\operatorname{codim}}
\newcommand\Alb{\operatorname{Alb}}
\newcommand\supp{\operatorname{supp}}
\newcommand\Bs{\operatorname{Bs}}
\begin{document}

\title{On the bicanonical map of irregular varieties}

\author[M.A. Barja]{Miguel Angel Barja}
\address{Departament de Matem\`atica Aplicada I, Universitat Polit\`ecnica de Catalunya,
ETSEIB Avda. Diagonal 647, 08028 Barcelona, Spain
 } \email{{\tt miguel.angel.barja@upc.edu
}}

\author[M. Lahoz]{Mart\'{\i} Lahoz}
\address{Departament d'\`Algebra i Geometria, Facultat de Matem\`atiques, Universitat
de Barcelona, Gran Via, 585, 08007 Barcelona, Spain}  \email{{\tt marti.lahoz@ub.edu
}}

\author[J. C. Naranjo]{\\Juan Carlos Naranjo}
\address{Departament d'\`Algebra i Geometria, Facultat de Matem\`atiques, Universitat
de Barcelona, Gran Via, 585, 08007 Barcelona, Spain}  \email{{\tt jcnaranjo@ub.edu
}}

\author[G. Pareschi]{Giuseppe Pareschi}
\address{Dipartimento di Matematica, Universit\`a di Roma, Tor Vergata, V.le della
Ricerca Scientifica, I-00133 Roma, Italy} \email{{\tt
pareschi@mat.uniroma2.it}}

\thanks{This work was completed while GP was visiting member of MSRI during the Program ``Algebraic Geometry'' (April 2009) and JCN was visiting the University of Pavia (Spring 2009, granted by the Spanish Ministerio de Ciencia e Innovaci\'on), and MAB was visiting the Centre de Recerca Matem\`atica.
GP thanks MSRI for great hospitality and excellent research atmosphere. JCN thanks the University of Pavia for his warm hospitality. MAB, ML and JCN were
partially supported by the Proyecto de Investigaci\'on MTM2006-14234. MAB and ML were also partially supported by 2005SGR-557 and JCN was partially supported by 2005SGR-787.
}

%\date{\today}
\maketitle

\setlength{\parskip}{.1 in}

\markboth{M.A. BARJA, M. LAHOZ, J.C. NARANJO, G. PARESCHI} {\bf On the bicanonical map of irregular varieties}

\begin{abstract} From the point of view of uniform bounds for the birationality of pluricanonical maps, irregular varieties of general type and maximal Albanese dimension behave similarly to curves. In fact Chen-Hacon showed that, at least  when their holomorphic Euler characteristic is positive, the tricanonical map of such varieties is always birational. In this paper we study the bicanonical map. We consider the natural subclass of varieties of maximal Albanese dimension formed by primitive varieties of Albanese general type. We prove that the only such varieties with non-birational bicanonical map  are the natural higher-dimensional generalization to this context of curves of genus $2$: varieties birationally equivalent to the theta-divisor of an indecomposable principally polarized abelian variety. The proof is based on the  (generalized) Fourier-Mukai transform.
\end{abstract}

\section{Introduction}

Pluricanonical maps are an essential tool for  understanding  varieties of general type. In particular, given a variety of general type $X$, it is important to know, or at least bound, the minimal integer $m_0(X)$ such that the  pluricanonical maps
$$\phi_m: X\dashrightarrow \PP(H^0(X,\omega_X^m)^*)$$
are birational onto their image for each $m\ge m_0(X)$.
In this paper we will  deal with  this sort of problems for complex \emph{irregular} varieties (i.e. varieties such that $q(X):=h^1(\OO_X)>0$) of general type, mostly of  \emph{maximal Albanese dimension} (\emph{m.A.d.} for short), i.e. such that  their Albanese map $alb:X\rightarrow \Alb X$ is generically finite. In a sense these are the most basic irregular varieties since the Stein factorization of the Albanese map provides a canonical fibration onto a normal variety of the same irregularity, whose smooth models have m.A.d.. From the point of view of the birationality of pluricanonical maps, Chen-Hacon (\cite{chen-hacon}) showed that
m.A.d. varieties of general type and $\chi(\omega_X)>0$ behave like curves: their tricanonical map  is always birational\footnote{
 m.A.d. varieties of general type with $\chi(\omega_X)=0$ were discovered by Ein-Lazarsfeld in \cite{el} and do not exist in dimension $\le 2$. However one knows, again by \cite{chen-hacon}, that their 6-canonical map is birational.}.
Such result have been made more precise in \cite{pp5}, where it is proved that if $X$ is \emph{non-special} (see below) then the tricanonical map is an embedding outside the exceptional locus of the Albanese map. It is worth to note that ample line bundles on abelian varieties have the same behavior, since the third power of an ample line bundle is always very ample. In fact the analogy has been explained in \cite{pp3} and references therein, where it is showed that the two facts can be proved in the same way.

 In this paper we address  the next natural problem: \emph{which are the m.A.d. varieties of general type
  whose  bicanonical map is not birational? }
    To describe our view of the question, some preliminary remarks about irregular varieties are in order. According to the seminal work of Green-Lazarsfeld (\cite{gl1,gl2}) and Ein-Lazarsfeld (\cite{el}),  key  invariants of irregular varieties are their \emph{cohomological support loci}
  $$V^i(\omega_X)=\{\alpha\in\Pico X\>|\> h^i(\omega_X\otimes\alpha)>0\}$$
  These are intimately related to the geometry of $X$ since, by the main theorem of \cite{gl2}, the positive dimensional components of $V^i(\omega_X)$  are translates of subtori arising, for $i>0$, only in presence of a morphism with connected fibres to a lower dimensional m.A.d. normal variety. It follows that, up to birational equivalence,  a variety such that $\dim V^i(\omega_X)>0$ for some $i>0$ has a morphism onto a smooth lower-dimensional m.A.d. variety $Y$ such that $\dim V^i(\omega_Y)=0$ for all $i>0$. This is the reason of the following
  \begin{introdeff}[Catanese, {\cite[Def. 1.24]{catanese}}] A smooth projective irregular variety such that $\dim V^i(\omega_X)=0$ for all $i>0$ is called primitive.
  \end{introdeff}
Moreover in a recent paper  Pareschi-Popa introduced a single numerical invariant measuring the \emph{codimension} of the various support loci, namely the \emph{generic vanishing index}
$$gv(X):=\min_{i>0}\{\codim _{\Pico X}V^i(\omega_X)-i\}.$$
In this terminology Green-Lazarsfeld's Generic Vanishing Theorem of \cite{gl1} can be rephrased as follows: \emph{ if $X$ has m.A.d. then $gv(\omega_X)\ge 0$. } The main feature of the invariant $gv(X)$ is that it  governs the local sheaf-theoretic properties of the Fourier-Mukai transform of the structure sheaf (\cite{pp6} Def.3.1 and Cor.3.2).
A variety verifying the extremal case $gv(X)=0$ has quite special properties. From the Fourier-Mukai point of view, $gv(X)=0$ means that the  transform of the structure sheaf  of $X$ has torsion (see Theorem \ref{GV1} below). From the geometric side,  it means that the image  $X$ via the Albanese map is fibered by subtori of the Albanese variety (\cite{el}, Proof of Theorem 3).
 For these reasons, m.A.d. varieties with $gv(\omega_X)=0$ will be referred to as \emph{special}.

  Consequently, we  are lead to divide m.A.d. varieties  into four disjoint subclasses: the \emph{non-special} varieties are further distinguished into \emph{primitive} and \emph{non-primitive} ones and similarly for the \emph{special} varieties. An immediate computation shows that for \emph{primitive} varieties, non-special (resp. special) mean simply $\dim X<q(X)$ (resp. $\dim X=q(X)$, i.e. the Albanese map is surjective). We recall that in the literature there is a specific definition also for m.A.d. varieties with $\dim X<q(X)$: they are called \emph{of Albanese general type} (Catanese, \cite[Def. 1.7]{catanese}). Therefore \emph{primitive non-special} is equivalent to \emph{primitive of Albanese general type}.

  Let us go back to the problem of describing  m.A.d. varieties of general type such that their bicanonical map is not surjective.
In dimension $1$  these are the curves of genus $2$. The problem of classifying  surfaces of general type whose bicanonical map is not birational has attracted considerable interest, starting from Du Val, later on Bombieri and more recently, Catanese, Ciliberto, Francia, Mendes Lopes, Pardini, Xiao Gang and others. We refer to the surveys \cite{ciliberto} and (more recent) \cite{bcpi} for an account on this work. Although the general classification is still not fully achieved, things are much better behaved for surfaces of maximal Albanese dimension. In fact, from Theorems 8, 9 and 10 of \cite{bcpi}
 one extracts the following result, due to the combined efforts of Catanese, Ciliberto and Mendes Lopes (\cite{CCiMe,CiMe}): minimal m.A.d. surfaces of general type  whose bicanonical map is not surjective are either fibered by curves of genus $2$    (this  is usually  referred to as \emph{the standard case}) or\\
  ($a$) the symmetric product of a curve of genus $3$;\\
  ($b$) the double \smallskip cover of  principally polarized abelian surface branched at a divisor $D\in |2\Theta|$.\\
Interestingly enough, the analogy with curves persists, as both cases ($a$) and ($b$) can be seen as 2-dimensional generalizations of curves of genus $2$: \\
  ($a$) curves of genus $2$ are the theta-divisors of indecomposable principally polarized abelian surfaces. Symmetric products of curves of genus $3$ are precisely the minimal surfaces which are birational to theta divisors of indecomposable principally polarized 3-folds. This generalizes to arbitrary dimension: the bicanonical map of a variety which is birational to a theta-divisor of an indecomposable p.p.a.v. has degree $2$, as it factors (up to birational equivalence) through the Kummer map.\\
 ($b$) A surface as in  ($b$) is the 2-dimensional analogue of the double cover of an elliptic curve branched at two points, i.e.  a bielliptic curves of genus $2$. Again, in arbitrary dimension a variety $X$ birational to the double cover of p.p.a.v $A$, branched at a (say smooth, for simplicity) divisor $B\in|2\Theta|$ has degree $2$, as it factors (up to birational equivalence) through the lifting to $X$ of the natural involution of  $A$.\smallskip\\
An immediate computation shows that the varieties described in ($a$) and ($b$) above are \emph{primitive} (more precisely, $V^i(\omega_X)=\{\hat 0\}$ for $i>0$, where $\hat 0$ denotes the identity point of $\Pico X$). The ones of ($a$) are \emph{non-special} while the ones of ($b$) are \emph{special}. As for curves, in both cases their holomorphic Euler characteristic is $\chi(\omega_X)=1$, the minimal value\footnote{ m.A.d. varieties of general type satisfy the sharp lower bound  $\chi(\omega_X)\ge 0$ (\cite{el}), but if $X$ is, in addition, \emph{primitive} then $\chi(\omega_X)\ge 1$ (see Proposition \ref{lazarsfeld} below for a more general statement).}. Concerning the standard case, m.A.d surfaces fibered in curves of genus $2$ are \emph{non-primitive}.  Products of  curves, $C\times D$, such that $g(C)\ge g(D)=2$   are \emph{non-special} examples. The surface of  \cite{CCiMe} Th.3.23 is an example of a \emph{special} m.A.d. surface presenting the standard case. For m.A.d varieties of arbitrary dimension, the natural  generalization of the notion of \emph{standard case} is a variety fibered by \emph{primitive} varieties whose bicanonical map is not surjective.

In this paper we will focus on  non-special primitive varieties. The somewhat surprising result is that, in arbitrary dimension,  the picture is the same:
   \begin{thmalpha}\label{A} Let $X$ be a smooth complex variety of maximal Albanese dimension and of  general type. Assume  moreover that $X$ is primitive and  that $\dim X< q(X)$.  The following are equivalent\\
  (a) the bicanonical map of $X$ is non-birational,\\
  (b)   $X$ is birationally equivalent to a theta-divisor of an indecomposable p.p.a.v..
  \end{thmalpha}

 Observe that the hypothesis of being primitive and  that $\dim X<q(X)$ could also be phrased as being primitive and non-special or primitive and of Albanese general type.

As a consequence of the  above quoted main theorem of \cite{gl2}, it follows
  \begin{coralpha}\label{B}
  Let $X$ be a smooth complex variety of general type with $\dim X<q(X)$. If   the bicanonical map of $X$ is non-birational then either $X$ has a morphism  onto a lower dimensional irregular normal variety  or $X$ is birational to the theta-divisor of an indecomposable p.p.a.v..
  \end{coralpha}

   Theorem \ref{A} leaves open the classification of m.A.d varieties of general type  with non-birational bicanonical map of the other three types. We conjecture  that: \\
   - if \emph{primitive and special} such varieties  should have $\chi(\omega_X)=1$ (but we don't have a clear idea about the possibility of other examples besides ($b$) above, and what they should be); \\
    - the  \emph{non-primitive} ones should all present the \emph{standard case} (see the above discussion). Among these
  the  \emph{non-special} ones should be birational to the product of  a non-special m.A.d. varieties and a theta-divisor.

  The method of proof of Theorem \ref{A} is different from the one of \cite{CCiMe} and \cite{CiMe}, even if restricted to the $2$-dimensional case. The basic framework is the (generalized) Fourier-Mukai transform  and its relation with generic vanishing (the necessary background material is reviewed in \S2). In particular,   the equivalence
  $$ gv(\mathcal{F})\ge 1\quad\Leftrightarrow \quad\widehat{R\Delta\mathcal{F}}\ \hbox{ torsion-free} $$
  (\cite{pp5,pp6}, see also Theorem \ref{GV1} below) is repeatedly used.

   In \S3 we prove a slight improvement of Hacon-Pardini's cohomological characterization of theta-divisors \cite[Thm. 2]{hp}. The result is Proposition \ref{theta1} below, stating that: \emph{primitive non-special varieties with $\chi(\omega_X)=1$ are birational to theta-divisors}\footnote{This result was proved independently (with a different proof) by Lazarsfeld-Popa \cite{lp}, building on the ideas of \cite{hp}.}. Interestingly, a slightly weaker version of the above result holds in any characteristic (we refer to Corollary \ref{charp} for the precise statement).

   In \S4, via the notion of \emph{continuous global generation} (\cite{pp1,pp3} and references therein), we prove a birationality criterion asserting, roughly speaking,   that \emph{the non-birationality of the bicanonical map implies that, for general $\alpha\in Pic^0X$, the linear series $|\omega_X\otimes\alpha|$ has a base divisor} (we refer to Theorem \ref{birprov} and Corollary \ref{birdef} for the precise statements). This works under hypotheses which are more general than those of Theorem \ref{A}. For example, it works also for primitive special varieties. Moreover, it is worth to note that the above mentioned analogy with ample line bundles on abelian varieties persists. In fact our birationality criterion, and its proof, are similar to Ohbuchi's theorem, asserting that twice an ample line bundle $L$ on an abelian variety is ample unless $L$ has a base divisor, as proved in \cite{pp2}.

  The above birationality criterion is used in  \S5 where, by means of a geometric analysis of the paracanonical system, combined with the Fourier-Mukai transform, we show that a variety $X$ satisfying ($a$) of Theorem \ref{A} has $\chi(\omega_X)=1$. At this point  Theorem \ref{A} is proved by applying  the above mentioned Hacon-Pardini's type characterization of theta-divisors.

  Finally, in the Appendix we provide the proof of an useful technical fact about Fourier-Mukai transform  that we couldn't find in the literature. It is our hope that some of the steps of the argument (namely, the birationality criterion of \S4, the decomposition of the Picard torus and  the explicit description of the Poincar\'e line bundle of \S5,  the result of the Appendix) will be of independent interest.

  Unless otherwise stated, throughout the paper the word \emph{variety} will mean \emph{smooth projective complex variety }(but the careful reader will notice that at most steps of the arguments are completely algebraic, and work in any characteristic). Given a line bundle $\alpha$ on a variety $X$, by abuse of language we will also denote $\alpha$ the point of the Picard variety of $X$ parametrizing $\alpha$.

  \noindent{\bf Acknowledgments. } We thank Mihnea Popa for numerous conversations. In fact his contribution in shaping  many tools used in this paper doesn't need to be acknowledged. More recently a question of his lead us to discover a mistaken argument in a previous version of this work. We thank also Rob Lazarsfeld and Roberto Pignatelli.

\section{(Generalized) Fourier-Mukai transform} Here we review  from \cite{mukai} and \cite{pp4,pp5, pp6} the material about (generalized) Fourier-Mukai transforms and generic vanishing that will be needed  in the sequel.

\subsection{(Generalized) Fourier-Mukai transforms and generic vanishing. }
\begin{term/not}[Fourier-Mukai and generic vanishing] \label{FMetc} Let $X$ be a  variety of dimension $d$, equipped with a morphism to a $n$-dimensional abelian variety $$a:X\rightarrow A.$$
Let ${\mathcal P}$ be a Poincar\'e line bundle on $A\times {\rm Pic}^0A $. We will denote
$$P_a=(a\times {\rm id}_{{\rm Pic}^0\,A})^*({\mathcal P})$$
When $a=alb$, the Albanese map of $X$,  then the map $alb^*$ identifies ${\rm Pic}^0({\rm Alb} X)$   to ${\rm Pic}^0X$ and the line bundle $P_{alb}$ is identified to the Poincar\'e line bundle of $X$. We will denote  $P=P_{alb}$.
Letting $p$ and $q$ the two projections of $X\times {\rm Pic}^0 A$,
 we consider the left-exact functor
 $$\Phi_{P_a}(\F)=q_*(p^*(\F)\otimes P_a)$$ and its derived functor
$${ R}\Phi_{P_a}:{\bf D}(X)\rightarrow {\bf D}({\rm Pic}^0A).$$
Sometimes we will have to consider  the analogous derived functor
${ R}\Phi_{P_a^\vee}:{\bf D}(X)\rightarrow {\bf D}({\rm Pic}^0A)$ as well. Since ${\mathcal P}^{-1}\cong (1_A\times (-1)_{\Pico A})^*{\mathcal P}$, we have that
$$  R\Phi_{P_a^\vee}= (-1_{\Pico A})^*R\Phi_{P_a}$$
Finally, given a coherent sheaf $\F$ on $X$, its \emph{i-th cohomological support locus  with respect to $a$} is
$$V^i_a(\F)=\{\alpha \in {\rm Pic}^0A\>|\>h^i(\F\otimes a^*\alpha)>0\}$$
Again, when $a$ is the Albanese map of $X$, we will omit the subscript, simply writing
\ $V^i(\F)$.
A natural measure of the size of the full package of the $V^i_a(\F)$'s is provided by the \emph{generic vanishing index} of $\F$ (with respect to $a$) (see \cite[Def. 3.1]{pp6})
$$gv_a(\F):=\min_{i>0}\{\codim _{\Pico A}V^i_a(\omega_X)-i\}.$$
\end{term/not}

\bigskip A first basic result relates the generic vanishing index  to the fact that the F-M transform of its Grothendieck dual is a sheaf (in cohomological degree $d$).
In what follows, we will adopt the following notation for the dualizing functor
$$R\Delta(\F)=R\,{\mathcal Hom}(\F,\omega_X).$$

\begin{thm}[{\cite[Thm. A]{pp4},\cite[Thm. 2.2]{pp6}}]\label{GV} The following are equivalent\\
(a) $gv_a(\F)\ge 0$;\\
(b) $R^i\Phi_{P_a}(R\Delta\F)=0$ for all $i\ne d$.
\end{thm}

\begin{term/not}[$GV$-sheaves]\label{gv}
If $gv_a(\F)\ge 0$ the sheaf $\F$  is said to be a \emph{$GV$-sheaf (generic vanishing sheaf).  } If this is the case, Theorem \ref{GV} says that the full transform ${\bf R}\Phi_{P_a}(R\Delta(\F))$ is a sheaf concentrated in degree $\dim X$:
 $${ R}\Phi_{P_a}(R\Delta\F)=R^{d}\Phi_{P_a}(R\Delta\F)[-d]$$
 Then one usually denotes
$$ R^{d}\Phi_{P_a}(R\Delta\F)=\widehat{R\Delta\F}$$
Note that, by (a) of Theorem \ref{GV}, $H^i(\F\otimes a^*\alpha)=0$ for all $i>0$ and general $\alpha\in \Pico A$. Therefore, by deformation-invariance of $\chi$, the generic value of $H^0(\F\otimes a^*\alpha)$ equals $\chi(\F)$. Since, by base-change, the fiber of $\widehat{R\Delta\F}$ at a general point $\alpha\in \Pico A$ is isomorphic to $H^{d}(R\Delta(\F)\otimes a^*\alpha)\cong H^0(\F\otimes a^*\alpha^{-1})^*$, the (generic) rank of $\widehat{R\Delta(\F)}$ is
\begin{equation}\label{rank}{\rm rk\,} (\widehat{R\Delta\F})=\chi(\F)\end{equation}
\end{term/not}
 Via base-change Theorem \ref{GV} yields
\begin{cor}[{\cite[Thm. 1.2]{hacon}, \cite[Prop. 3.13]{pp4}}]\label{inclusions}  If $gv_a(\F)\ge 0$ then
$$V^d_a(\F)\subseteq \cdots\subseteq  V^1_a(\F)\subseteq V^0_a(\F)$$
\end{cor}

 From Grothendieck duality and Theorem \ref{GV} it follows
\begin{cor}[{\cite[Rem. 3.12]{pp4}, \cite[Pf. of Cor. 3.2]{pp6}}]\label{duality}  If $gv_a(\F)\ge 0$  then
$$  {\mathcal Ext}^i_{\OO_{{\rm Pic}^0A}}(\widehat{R\Delta\F},\OO_{\Pico A})\cong R^i\Phi_{P_a^\vee}(\F)\cong (-1_{\Pico A})^*R^i\Phi_{P_a}(\F)\  $$
\end{cor}

In \cite{pp6} Pareschi-Popa established a dictionary between the value of $gv_a(\F)$ and the local properties of the transform $\widehat{R\Delta(\F)}$. Its first instance  is the following

\begin{thm}[{\cite[Cor. 3.2]{pp6}}] \label{GV1}
 Assume that $\F$ is $GV$ (with respect to $a$). Then the following are equivalent\\
(a) $gv_a(\F)\ge 1$;\\
(b) $\widehat{R\Delta\F}$ is a torsion-free sheaf.
\end{thm}

In the above terminology, the \emph{Generic Vanishing Theorem} of Green-Lazarsfeld (\cite{gl1}, see also \cite[Rem. 1.6]{el}) asserts, in particular, that  $gv_a(\omega_X)\ge 0$  if $a$ is generically finite. The following Proposition
shows that also the converse is true
\begin{prop}\label{finiteness} Assume that $gv_a(\omega_X)\ge 0$. Then $a$ is generically finite.
\end{prop}
More generally, one can prove in the same way the converse of the full Generic Vanishing theorem of \cite{gl1}, namely that if $gv(\omega_X)\ge -k$ then $\dim a(X)\ge d-k$. Such statement was proved independently by Lazarsfeld-Popa \cite[Prop. 1.5]{lp}.
\begin{proof} Let $e=\dim X-\dim a(X)$. Let $$X\buildrel{b}\over\rightarrow Y\buildrel{c}\over\rightarrow A$$ be the Stein factorization of $a$. Since the $V^i_a(\omega_X)$ are birational invariants, we can assume that $Y$ is smooth. By Koll\'ar's theorem \cite[Thm. 3.1]{kollar2}, the Leray spectral sequence of $b$ splits:
 \begin{equation}\label{kollar1}H^i(\omega_X\otimes b^*(c^*\alpha))\cong\bigoplus_{k=0}^iH^k(R^{i-k}b_*(\omega_X)\otimes c^*\alpha)
 \end{equation}
 Moreover, again by a result of Koll\'ar \cite[Prop. 7.6]{kollar1},
 \begin{equation}\label{kollar2} R^eb_*(\omega_X)=\omega_Y
 \end{equation}
 Assume that $\dim Y<\dim X$, i.e. $e>0$. From (\ref{kollar1}) for $i=e$ and (\ref{kollar2}) it follows that $V^e_a(\omega_X)$ contains $V^0_c(\omega_Y)$. The fact that $gv_a(\omega_X)\ge 0$  implies that $\codim V^0_c(\omega_Y)\ge e>0$. Since $c:Y\rightarrow A$ is generically finite, the above mentioned Green-Lazarsfeld Generic Vanishing Theorem yields that  $gv_c(\omega_Y)\ge 0$. Hence, by Theorem \ref{GV}, $\Phi_{P_c}(\OO_Y)$ is a sheaf (in cohomological degree equal to $\dim Y$), denoted $\widehat{\OO_Y}$. Since $V^0_c(\omega_Y)$ is a \emph{proper} subvariety of $\Pico A$, $\widehat{\OO_Y}$ must be a torsion sheaf. Therefore, by Theorem \ref{GV1}, there is a $i>0$ such that $\codim_{\Pico A}V^i_c(\omega_Y)=i$. Since, again by (\ref{kollar1}) and (\ref{kollar2}),  $V^i_c(\omega_Y)$ is contained in $V^{e+i}_a(\omega_X)$, it follows that
 $\codim_{\Pico A}V^{e+i}_a(\omega_X)<e+i$, a contradiction.
 \end{proof}

\subsection{Mukai's equivalence of derived categories of abelian varieties} Assume that $X$ coincides with the abelian variety $A$ (and the map $a$ is the identity).  In this special case, according to Notation \ref{FMetc}, ${\mathcal P}$ denotes  the
 Poincar\'e line bundle on $A\times \Pico A$.  Then Mukai's theorem asserts that $R\Phi_{\mathcal P}$ is an equivalence of categories.
More precisely, denoting $R\Psi_{\mathcal P}:{\bf D}(\Pico A)\rightarrow {\bf D}(A)$ and $n=\dim A$:
\begin{thm}[{\cite[Thm. 2.2]{mukai}}]\label{mukai}  $$R\Psi_{\mathcal P}\circ R\Phi_{\mathcal P}=(-1)^*_A[n],\qquad R\Phi_{\mathcal P}\circ R\Psi_{\mathcal P}=
(-1)^*_{\Pico A}[n].$$
\end{thm}
Given a morphism $a:X\rightarrow A$, the functors $R\Phi_{P_a}$ and $R\Phi_{\mathcal P}$ (see Notation \ref{FMetc}) are related by the following formula
\begin{prop}\label{composition}
 $$R\Phi_{P_a} \cong R\Phi_{\mathcal P}\circ Ra_*$$
\end{prop}
\begin{proof}  \ \ $R\Phi_{P_a} (\>.\>)= R{q}_*(p_X^*(\>.\>)\otimes (a\times {\rm id})^*{\mathcal P})\buildrel{L+PF}\over\cong  R{q}_*(R (a\times {\rm id})_*(p_X^*(\>.\>))\otimes{\mathcal P}) \cong$ $$\buildrel{BC}\over\cong R{q}_*(p_A^*(Ra_*(\>.\>))\otimes {\mathcal P})=R\Phi_{\mathcal P}\circ Ra_*(\>.\>)$$
where: $L=$ Leray, $PF=$ projection formula and $BC=$ Base Change.
\end{proof}
\begin{rem}\label{alg} Notice that, with the exception of Proposition \ref{finiteness} (which uses Koll\'ar theorems on higher direct images of dualizing sheaves), all the results of the present section are algebraic, and work on algebraically closed fields of any characteristic.
\end{rem}

\section{On a cohomological characterization of theta-divisors}
As a simple but instructive application of the results reviewed in the previous section, we prove the following slight improvement of Hacon-Pardini's cohomological characterization of theta-divisors \cite[Prop. 4.2]{hp}. Proposition \ref{theta1} below is proved independently, with a different proof, also in \cite[Prop. 3.13]{lp}. An algebraic version, valid in any characteristic, is provided by Corollary \ref{charp} below.
\begin{prop}\label{theta1} Let $X$ be a $d$-dimensional smooth projective variety (or a compact K\"ahler manifold) such that:
(a) $\dim V^i(\omega_X)=0$ for all $i>0$; (b) $d=\dim X<q(X)$ \emph{(i.e., in the terminology of the introduction, $X$ is primitive and non-special)}, and (c) $\chi(\omega_X)=1$. Then $\Alb X$ is a principally polarized abelian variety and the Albanese map $alb:X\rightarrow \Alb X$ maps $X$ birationally onto a theta-divisor.
\end{prop}
\begin{proof} We denote $q=q(X)$. By hypothesis ($a$) and ($b$)  $gv(\omega_X)\ge 1$. Therefore, by Theorem \ref{GV1}, $\widehat{\OO_X}$ is torsion-free. Since, by (\ref{rank}), ${\rm rk\,}\widehat{\OO_X}=\chi(\omega_X)\buildrel{(c)}\over =1$,  we get that $\widehat{\OO_X}$ is an ideal sheaf twisted by a line bundle of $\Pico X$:
$$\widehat{\OO_X}=\I_Z\otimes L$$
By base change, the support of $Z$ is contained in the union of the $V^i(\omega_X)$ for $i>0$ which are assumed, by ($a$), to be finite sets. Therefore ${\mathcal E}xt^{i}(\widehat{\OO_X},\OO_{\Pico X})={\mathcal E}xt^{i+1}(\OO_Z,\OO_{\Pico X})=0$ for $i+1\ne q$. On the other hand, by Proposition \ref{appendix} of the Appendix, and Corollary \ref{duality} it follows that
\begin{equation}\label{ext}
{\mathcal E}xt^{d}(\widehat{\OO_X},\OO_{\Pico X})\cong(-1_{\Pico X})^*R^d\Phi_P(\omega_X)\cong\CC(\hat 0).
\end{equation} This implies:\\
($a$)  $d=q-1$. \\
($b$) $\OO_Z=\CC(\hat 0)$ (Indeed ${\mathcal E}xt^{i}(\CC(\hat 0),\OO_{\Pico X})$ is zero for $i<q(X)$ and equal to $\CC(\hat 0)$ for $i=q(X)$.  Since $R{\mathcal H}om(\cdot,\OO_{\Pico X})$ is an involution, ($b$) follows from (\ref{ext})). In conclusion
$$\widehat{\OO_X}=\I_{\hat 0}\otimes L,$$ where $L$ is a line bundle on $\Pico X$ and $\I_{\hat 0}$ is the ideal sheaf of the (reduced) point $\hat 0$.
By Proposition \ref{composition}
$$R\Phi_P(\OO_X)=R\Phi_{\mathcal P}(R\,alb_*\OO_X)=\I_{\hat 0}\otimes L[-q+1]$$
Therefore, by Mukai's Inversion Theorem \ref{mukai}
\begin{equation}\label{inversion}
R\Psi_{\mathcal P}(\I_{\hat 0}\otimes L)=(-1)^*_{\Pico X}R\,alb_*\OO_X[-1]
\end{equation}
In particular:
\begin{equation}\label{princ}R^0\Psi_{\mathcal P}(\I_{\hat 0}\otimes L)=0\qquad\hbox{
and }\qquad R^1\Phi_{\mathcal P}(\I_{\hat 0}\otimes L)\cong alb_*\OO_X
\end{equation}
Applying $\psi_{\mathcal P}$ to the standard exact sequence
\begin{equation}\label{stan}
0\rightarrow \I_{\hat 0}\otimes L\rightarrow L\rightarrow \OO_{\hat 0}\otimes L\rightarrow 0,
\end{equation}
and using (\ref{princ}) we get,
\begin{equation}\label{basta!}
0\rightarrow R^0\Psi_{\mathcal P}(L)\rightarrow \OO_{\Alb X}\rightarrow alb_*\OO_X
\end{equation}
whence
$R^0\Psi_{\mathcal P}(L)$ is supported everywhere (since $alb_*\OO_X$ is supported on a divisor). It is well known that this implies that $L$ is \emph{ample}. Therefore $R^i\Psi_{\mathcal P}(L)=0$ for $i>0$. Therefore, by sequence (\ref{stan}), $R^i\Psi_{\mathcal P}(\I_{\hat 0}\otimes L)=0$ for $i>1$. By  (\ref{inversion}) and (\ref{princ}), this implies that $R^{i}alb_*(\OO_X)=0$ for $i>0$. Furthermore, (\ref{basta!}) implies easily that $h^0(L)=1$, i.e. $L$ is a \emph{principal} polarization. Therefore, via the identification $\Alb(X)\cong \Pico(X)$ provided by $L$, we have  $R^0\Psi_{\mathcal P}(L)\cong L^{-1}$ (see \cite[Prop. 3.11(1)]{mukai}). Since the arrow on the right in (\ref{basta!}) is onto, it follows that $alb_*\OO_X=\OO_D$, where $D$ is a divisor in $|L|$. Since we already know that $alb$ is generically finite (Prop. \ref{finiteness}), this implies that $alb$ is a birational morphism onto $D$.
\end{proof}
Note that the proof is entirely algebraic, except for the use of Prop. \ref{finiteness}, (see Remark \ref{alg}). Therefore, the following statement holds
\begin{cor}\label{charp} let $X$ be a smooth projective variety (over any algebraically closed field) such that: (a) $\dim V^i(\omega_X)=0$ for all $i>0$; (b) $d=\dim X<\dim \Alb X$ and the Albanese map of $X$ is generically finite,  and (c) $\chi(\omega_X)=1$. Then $\Alb X$ is a principally polarized abelian variety and the Albanese map $alb:X\rightarrow \Alb X$ maps $X$ birationally onto a theta-divisor.
\end{cor}

\section{Birationality criterion}
First we review   the necessary background about  the notion of \emph{continuous global generation}. Then we apply such machinery to the canonical bundle of a variety, under some hypothesis which apply to the main theorem of the present paper. Finally,  by applying the same machinery to the ideal sheaf of a  point twisted by the canonical line bundle, we prove the birationality criterion we are aiming for, i.e. Theorem \ref{birprov} and its Corollary \ref{birdef}.
\subsection{Continuous global generation}
 The notion of \emph{continuous global generation} is a useful notion related to global generation. It was introduced by Pareschi-Popa in \cite{pp1} and applied to various geometric problems in \cite{pp1,pp2,pp5}. See also the survey \cite{pp3}.

  Let $\F$ be a coherent sheaf on $X$, and let $T$ be a subset of ${\rm Pic}^0A$. Then we have the \emph{continuous evaluation map associated to the pair $(\F,T)$}:
$$ev_{T,\F}:\bigoplus_{\alpha\in T}H^0(X,\F\otimes a^*\alpha^{-1})\otimes a^*\alpha\rightarrow \F$$

\begin{deff}\label{CB}
 Let $p$ be a point of $X$. The sheaf $\F$ is said to be
\emph{continuously globally generated at $p$ (CGG at $p$ for short), with respect to the morphism  $a$}, if the map $ev_U$ is surjective \emph{at $p$}, for all non-empty Zariski open subsets $U\subseteq \Pico A$. When possible, we will omit the reference to the morphism $a$, and say simply $CGG$ at $p$.
\end{deff}

\begin{rem} \label{useful}
 If $L$ is a line bundle, then $L$ is not $CGG$ at $p$ if and only if there us a non-empty Zariski open set $V\subseteq  \Pico A$ such that $p$ is  a base point of $L\otimes a^*\alpha^{-1}$ for all $\alpha\in V$. This is equivalent to the fact that $p$ is a base point of $L\otimes a^*\alpha^{-1}$ for all $\alpha\in\Pico A$ such that $h^0(L\otimes a^*\alpha^{-1})$ is minimal.
\end{rem}

We will also need the following weaker version. This is in fact a variant of the \emph{weak continuous global generation} of \cite{pp2}. However, for technical reasons we prefer to give the following definition, which is natural in view of Proposition \ref{CGG-GG}(b) and Theorem \ref{regularity}(b) below.

\begin{deff}\label{ACB}
Let $p$ be a point of $X$ and let $\F$ be a GV-sheaf. Let $\widehat{R\Delta(\F)}$ be the transform of the dual of $\F$ and let $\tau=\tau(\widehat{R\Delta(\F)})$ be its torsion sheaf. Then $\F$ is said to be
\emph{essentially continuously globally generated at $p$ (ECGG at $p$ for short), with respect to the morphism  $a$}, if  the map $ev_T$ is surjective \emph{at $p$}, for all subsets of the form $T=U\cup S$, where $U$ is a non-empty Zariski open subset of $\Pico A$ and $S$ is the underlying subset of $\supp \tau$. As above, when possible, we will omit the reference to the morphism $a$, and say simply $ECGG$ at $p$.
\end{deff}

Obviously $CGG$ at $p$ implies $ECGG$ at $p$. Moreover,
we will say simply that a sheaf is $CGG$ (resp. $ECGG$), when it is $CGG$ (resp. $ECGG$) for all $p$.

The structure sheaf of an abelian variety $X$ is an easy example of a line bundle which is not $CGG$ at any point, since for any open subset $U$ of $\Pico X$ not containing the identity point $\hat 0$, the map $ev_{U,\OO_X}$ is zero, but it is $ECGG$. In fact it is well known that in this case $\widehat{\OO_X}=\CC (\hat 0)$ (\cite{mukai}). Hence the underlying subset of the support of $\tau(\widehat{\OO_X})=\CC (\hat 0)$ is the identity point $\{\hat 0\}$, and all evaluation maps $ev_{U\cup\{\hat 0\},\OO_X}$ are trivially surjective at all points. A generalization of this example is provided by Corollary \ref{omega-acgg}(b) below.

 A useful relation between continuous global generation and the usual \emph{global generation}  is  provided by the following, where, given a line bundle $M$, $\Bs(M)$ will denote its base locus.

\begin{prop}[{\cite[Prop. 2.4]{pp2}}]\label{CGG-GG}
Let $\F$ and $L$ be respectively a coherent sheaf and a line bundle on $X$, and let $p$ be a point of $X$. \\
(a) If both $\F$ and \ $L$ are continuously globally generated at $p$ then $\F\otimes L\otimes a^*\beta$ is globally generated at $p$ for any $\beta\in \Pico A$.\\
(b) Assume that $L$ is $CGG$ at $p$ and that $\F$ is $ECGG$ at $p$. Let $\tau(\widehat{R\Delta(\F)})$ be the torsion sheaf of $\widehat{R\Delta(\F)}$, and assume that the underlying set $S$ of the support of $\tau(\widehat{R\Delta(\F)})$ is finite. For any  $\beta\in \Pico A$, if $p\not\in \bigcup_{\alpha\in S_{+\beta}} \Bs(L\otimes a^*\alpha),$ then $\F\otimes L\otimes a^*\beta$ is globally generated at $p$.
\end{prop}
\begin{proof} (a) By Remark \ref{useful} there is an open subset $V_p$ of $\Pico A$ such that $p$ is not a base point of $L\otimes a^*\alpha$ for all $\alpha\in V_p$, i.e. the evaluation map at $p$: $H^0(L\otimes a^*\alpha)
\rightarrow  (L\otimes a^*\alpha)_p$ is surjective for all $\alpha\in V_p$. Since $\F$ is $CGG$ at $p$, it follows that the map
$$ev_{V_p}:\oplus_{\alpha\in V_p}H^0(\F\otimes a^*\beta\otimes a^*\alpha^{-1})\otimes H^0(L\otimes a^*\alpha)\rightarrow (\F\otimes L\otimes a^*\beta)_{p}$$ is surjective. This proves the assertion, since the above map factors trough $H^0(\F\otimes L\otimes a^*\beta)$.\\
(b) The proof is the same with the difference that now if we use the continuous evaluation map $ev_{T_p}$ with $T_p=V_p\cup (S_{+\beta})$. If $p\not\in \bigcup_{\alpha\in S_{+\beta}} \Bs(L\otimes a^*\alpha)$, then  the evaluation map $H^0(L\otimes a^*\alpha) \rightarrow (L\otimes a^*\alpha)_p$ is surjective also for all $\alpha\in S_{+\beta}$.
\end{proof}

 As for the usual global generation, in many applications it is useful to have a criterion ensuring that, if  the higher cohomology of a given sheaf $\F$ satisfies certain vanishing conditions,  then $\F$ is $CGG$ or $ECGG$. The following  criterion, which applies to \emph{sheaves on abelian varieties},
 is due to Pareschi-Popa. In the reference a sheaf $\F$ on an abelian variety $A$ such that $gv(\F)\ge 1$ is called \emph{$M$-regular}. The first part is of the following theorem is \cite[Prop. 2.13]{pp1} or \cite[Cor. 5.3]{pp5}. The proof of the second part essentially follows the proof of \cite[Thm. 4.1]{pp2}.

\begin{thm}\label{regularity}  Let $\F$ be a sheaf on an abelian variety $A$.\\
\noindent (a) If $gv(\F)\ge 1$ then $\F$ is $CGG$.\\
\noindent (b) If $\F$ is a GV-sheaf and $\supp \tau(\widehat{R\Delta(\F)})$ is a reduced scheme\footnote{We consider the annihilator support, instead of the Fitting support \cite{eis}.}, then $\F$ is $ECGG$.
\end{thm}

The main point is the following

\begin{lem}\label{torsion} Let $\F$ be a GV-sheaf on an abelian variety $A$. Let $\widehat{R\Delta(\F)}$ be the transform of the dual of $\F$ and let $\tau(\widehat{R\Delta(\F)})$ be its torsion sheaf.  Let $L$ be an ample line bundle on $A$. Then, for all sufficiently high $n\in\NN$, and for any subset $T\subseteq \Pico A$, the Fourier-Mukai transform $\Phi_{\mathcal P}$ induces a canonical isomorphism
$$H^0(A,\coker ev_{T,\F}\otimes L^n)\cong (\ker \psi_{T,\F})^*,$$
where $\psi$ is the natural evaluation map,
\begin{equation}\label{eq:psi}
\psi_{T,\F}:{\rm Hom} (\widehat{L^n},\widehat{R\Delta\F})\rightarrow\prod_{\alpha\in T} {\mathcal H}om (\widehat{L^n},\widehat{R\Delta\F})\otimes \CC(\alpha).
\end{equation}
\end{lem}
\begin{proof} Let $T\subseteq \Pico A$ be any subset. The map $H^0(ev_T\otimes L^n)$ is the ``continuous multiplication map of global sections'':
$$m^T_{\F,L^n}: \bigoplus_{\alpha\in T}H^0(\F\otimes a^*\alpha^{-1})\otimes H^0(L^n\otimes a^*\alpha)\rightarrow H^0(\F\otimes L^n).$$
A standard argument with Serre vanishing  shows that, if $n$ is big enough,
\begin{equation*}H^0(\coker (ev_{T,\F})\otimes L^n)\cong \coker (m^T_{\F,L^n})
\end{equation*}
By Serre duality, the dual of $m^T_{\F,L^n}$ is
\begin{equation}\label{dual}
{\rm Ext}^q(L^n,R\Delta\F)\rightarrow \prod_{\alpha\in T} {\rm Hom}_{\CC}\bigl(H^0(L^{n}\otimes a^*\alpha), H^q((R\Delta\F)\otimes a^*\alpha)\bigr)
\end{equation}
Let us interpret such map via the Fourier-Mukai transform. Concerning the source,  Mukai's Theorem \ref{mukai}  provides the isomorphism
$${\rm Hom}_{{\bf D}(A)}(L^n,R\Delta\F[q])\cong
{\rm Hom}_{{\bf D}(\Pico A)}(R\Phi_{\mathcal P}(L^n),R\Phi_{\mathcal P}(R\Delta\F)[q])\cong {\rm Hom}_{{\bf D}(\Pico A)}(\widehat{L^n},\widehat{R\Delta\F})$$
(here $q=\dim A$) i.e., since $\widehat{L^n}$ is locally free,
\begin{equation}\label{perseval}
{\rm Ext}^q_A(L^n,R\Delta\F)\cong {\rm Hom}_{\Pico A}(\widehat{L^n},\widehat{R\Delta\F}).
\end{equation}
Note that, besides $\widehat{L^n}$, also $\widehat{R\Delta(\F)}$ has the base change property, since $H^{d+1}(\F\otimes a^*\alpha)=0$ for all $\alpha\in\Pico A$, see \cite[Cor. 3, p. 53]{ab}.  This means that, in the target of the map (\ref{dual}), we have that  $H^0(L^{n}\otimes a^* \alpha)$ (respectively $ H^q((R\Delta\F)\otimes a^*\alpha))$) are isomorphic to the fiber at the point $\alpha$ of  $\widehat{L^{n}}$ (resp.  of $ \widehat{R\Delta\F}$). Hence   also the sheaf ${\mathcal H}om (\widehat{L^n},\widehat{R\Delta\F})$ has the base change property
and the Fourier-Mukai isomorphism (\ref{perseval}) identifies the map (\ref{dual}) to the evaluation map of the sheaf ${\mathcal H}om (\widehat{L^n},\widehat{R\Delta\F})$ at points in $T$:
$${\rm Hom} (\widehat{L^n},\widehat{R\Delta\F})\rightarrow\prod_{\alpha\in T} {\mathcal H}om (\widehat{L^n},\widehat{R\Delta\F})\otimes \CC(\alpha).$$\end{proof}

Now we are ready to proof the Theorem
\begin{proof}[Proof of Theorem \ref{regularity}] (a) By Theorem \ref{GV} we can assume that $\widehat{R\Delta \F}$ is torsion free. Then the evaluation map $\psi_{U,\F}$ is injective for \emph{all} open subsets $U\subseteq \Pico A$, so $H^0(\coker ev_{U,\F}\otimes L^n)=0$ for $n\gg0$. From Serre's theorem it follows that $\coker ev_{U,\F}=0$. \\

\noindent (b) By Nakayama's Lemma, given a non-zero global section $s\in {\rm Hom} (\widehat{L^n},\widehat{R\Delta\F})$, we have that $s(\alpha)\in {\mathcal H}om (\widehat{L^n},\widehat{R\Delta\F})\otimes \CC(\alpha)$ vanishes for all $\alpha$ in an dense subset $T\subset \Pico A$, only if $T$ does not meet a component of the support of the torsion part $\tau(\widehat{R\Delta(\F)})$ or this support is non-reduced. The second possibility is excluded by hypothesis.
And the first possibility is excluded if we consider subsets of the form $T=U\cup S$, where $S$ is the underlying set of $\supp \tau(\widehat{R\Delta(\F)})$. Hence, the map $\psi_{T,\F}$ is injective. Therefore $H^0(\coker ev_{T,\F}\otimes L^n)=0$ and, by Serre's theorem, $\coker ev_{T,\F}=0$.
\end{proof}

\subsection{The canonical bundle}
Throughout the rest of the present section we will work under the following
\begin{hyp}\label{hyp}
 Let $X$ be a variety of dimension $d$, equipped with a morphism to an abelian variety  \ $a:X\rightarrow A$ \  such that:\\
 (a) $\codim V^i_a(\omega_X)\ge i+1$ for all $i$ such that $0<i<d$ \emph{(note that this hypothesis is weaker than $gv(\omega_X)\ge 1$)};\\
 (b) the map $a^*\Pico A\rightarrow \Pico X$ is an embedding.
 \end{hyp}

   Hypothesis (b) yields that $V^d_a(\omega_X)=\{\hat 0\}$, where $\hat 0$ denotes the identity point of $\Pico A$. Moreover Hypothesis (a) is more than enough to ensure that $gv_a(\omega_X)\ge 0$. Therefore, by Theorem \ref{GV}, the transform of $R\Delta(\omega_X)=\OO_X$ is the sheaf $\widehat{\OO_X}$ concentrated in degree $d$. Moreover  the morphism $a$ is generically finite by Proposition \ref{finiteness}. We remark that Hypothesis (a) is equivalent to $gv_a(\omega_X)\ge 1$ unless $\dim X=\dim A$, i.e. the morphism $a$ is surjective. Therefore, by Theorem \ref{GV1}, the sheaf $\widehat{\OO_X}$ has torsion if and only if $\dim X=\dim A$, i.e. the map $a:X\rightarrow A$ is surjective.

\begin{prop}\label{trace} Under Hypotheses \ref{hyp}, assume moreover that $\dim X=\dim A$, i.e. that the morphism $a$ is surjective. Then the torsion $\tau(\widehat{\OO_X})$ is isomorphic to the one-dimensional skyscraper sheaf $\CC(\hat 0)$
\end{prop}

\begin{proof} The argument is  similar to the proof of Proposition 2.8 of \cite{pp5}.  Since $\Pico A$ is smooth, the functor
$R{\mathcal H}om(\> \cdot\>,\OO_X)$ is an involution on ${\bf D}(X)$.
Thus there is a spectral sequence
$$E^{i,-j}_2 := \mathcal{E}xt^i \Bigl((\mathcal{E}xt^j (\widehat{\OO_X}, \OO_{\Pico A}),\OO_{\Pico A}\Bigr)
\Rightarrow H^{i -j} = \mathcal{H}^{i-j} \widehat{\OO_X}  = \begin{cases}\widehat{\OO_X}
&\hbox{if $i= j$}\cr 0&\hbox{otherwise}\cr\end{cases}.$$ By duality (Corollary \ref{duality}), $\dim{\rm supp}({\mathcal Ext}^i\bigl(\widehat{\OO_X},\OO_{\Pico A})\bigr)=\dim\supp R^i\Phi_{P_a}(\omega_X)$ and, by base-change, $\supp R^i\Phi_{P_a}(\omega_X)\subseteq V^i_a(\omega_X)$. Thus Hypothesis \ref{hyp}(a) yields that for all $i$ such that $0<i<d$ we have $ {\rm
codim} ~\supp \bigl(\mathcal{E}xt^i(\widehat{\OO_X},\OO_{\Pico A})\bigr) > i $. Therefore
  $\mathcal{E}xt^i \Bigl(\mathcal{E}xt^j
(\widehat{\OO_X}, \OO_{\Pico A}),\OO_{\Pico A} \Bigr) = 0$ for all $(i,j)$ such that
$i - j \le 0$,  except for $(i,j)=(0,0)$ and $(i,j)=(d,d)$. Hence the only non-zero $E^{i,-i}_{\infty}$ terms are $E^{d,-d}_\infty$ and $E^{0,0}_{\infty}$, and we have the exact sequence
\begin{equation}\label{basic}
0\rightarrow E^{d,-d}_\infty\rightarrow H^0=\widehat{\OO_X}\rightarrow E^{0,0}_\infty\rightarrow 0
\end{equation}
The differentials coming into
$E^{0,0}_p$ are always zero, so we get an isomorphism,
$$ E^{0,0}_{\infty} \subseteq E^{0,0}_2=\widehat{\OO_X}^{**}$$
and (\ref{basic}) is identified to the canonical exact sequence
$$0\rightarrow \tau(\widehat{\OO_X})\rightarrow \widehat{\OO_X}\rightarrow \widehat{\OO_X}^{**}$$
Hence the torsion $\tau(\widehat{\OO_X})$ is canonically isomorphic to $E^{d,-d}_\infty$.
By Proposition \ref{appendix} of the Appendix, $R^d\Phi_{P_a}(\omega_X)\cong \CC(\hat 0)$. Hence $E^{d,-d}_2={\mathcal E}xt^d(R^d\Phi_{P_a}(\omega_X),\OO_{\Pico A})\cong {\mathcal E}xt^d(\CC(\hat 0),\OO_{\Pico A})=\CC(\hat 0)$. Since the differentials going out of $E^{d,-d}_p$ are always zero,  we get a surjection
$$E^{d,-d}_{2}\cong \CC(\hat 0)\rightarrow E^{d,-d}_\infty\cong\tau(\widehat{\OO_X})\rightarrow 0$$
As we already know that $\tau(\widehat{\OO_X})\ne 0$, it follows that $\tau(\widehat{\OO_X})\cong \CC(\hat 0)$.
\end{proof}

\begin{rem} It is easily seen that the injection $\CC (\hat 0)\hookrightarrow \widehat{a_*\OO_X}$ is $R^q\Phi_{\mathcal P}$ of the natural injection $\OO_A\rightarrow a_*\OO_X$.
\end{rem}

As a first consequence we get
\begin{prop}\label{lazarsfeld} The following are equivalent:\\
(a)  Hypothesis \ref{hyp} holds, and $X$ is not of general type;\\
(b) Hypothesis \ref{hyp} holds, and  $\chi(\omega_X)=0$; \\
(c) the morphism $a:X\rightarrow A$ is birational.
\end{prop}
\begin{proof} $(a)\Rightarrow (b)$:   The morphism $a$ is generically finite (Prop. \ref{finiteness}). In this case it is well known that $\chi(\omega_X)\ge 1$ implies that $X$ is of general type (by \cite{chen-hacon}, one knows even that the tricanonical map of $X$ is birational). \\
 $(b)\Rightarrow (c)$: since $\chi(\omega_X)=0$ then, by (\ref{rank}), $\widehat{\OO_X}$ is a torsion sheaf. By Theorem \ref{GV1} it follows that $gv(\omega_X)=0$. We already know that this, together with Hypothesis \ref{hyp}, is equivalent to the fact that the morphism $a$ is surjective. Therefore, by Proposition \ref{trace}, $\widehat{\OO_X}=\CC(\hat 0)$. By Proposition \ref{composition}
$$\CC(\hat 0)=\widehat{\OO_X}=\widehat{Ra_*\OO_X}$$
where the hat on the left is the transform $R\Phi_{P_a}$ (from $X$ to $\Pico A$) and the hat on the right
is the transform $R\Phi_{\mathcal P}$ (from $A$ to $\Pico A$). By Mukai's Theorem \ref{mukai}
$$\OO_A=R\Psi_{\mathcal P}(\CC(\hat 0))=R\Psi_{\mathcal P}(R\Phi_{\mathcal P}(Ra_*\OO_X))\buildrel{\ref{mukai}}\over=(-1)_A^*Ra_*\OO_X$$
whence $Ra_*\OO_X=\OO_A$. In particular, $a$ has degree $1$.
\end{proof}

Another useful consequence of the previous analysis is the following. We recall that
in the first case ($gv(\omega_X)\ge 1$) everything is essentially  already known \cite[Prop. 5.5]{pp5}.

\begin{cor}\label{omega-acgg}
 Assume Hypothesis \ref{hyp} and that $X$ is of general type.\\ (a) If $\dim X<\dim A$ then $a_*\omega_X$ is CGG.\\
 (b) If $\dim X=\dim A$, i.e. $a$ is surjective, then $a_*\omega_X$ is ECGG, with $\{\hat 0\}$ as underlying subset of $\supp \tau(R\Delta(a_*\omega_X))$.
\end{cor}
 \begin{proof} By Grauert-Riemenschneider vanishing (and projection formula), $R^ia_*(\omega_X\otimes a^*\alpha)=0$ for $i>0$. Therefore the Leray spectral sequence degenerates giving
 \begin{equation}\label{astar}
 V^i_a(\omega_X)=V^i(a_*\omega_X).
 \end{equation}
 (a) If $\dim X<\dim A$ then $gv(a_*\omega_X)\ge 1$. Therefore it is $CGG$ by Theorem \ref{regularity}(a). As remarked before, this part of the result is well known \cite[Prop. 5.5]{pp5}.\\
 (b) If $\dim X=\dim A$ then $\tau(\widehat{\OO_X})\cong \CC(\hat0)$ (Proposition \ref{trace}). Hence, by Theorem \ref{regularity}(b), $a_*\omega_X$ is $ECGG$.
 \end{proof}

\subsection{Birationality criterion. }
We introduce a last piece of notation.
\begin{term/not}\label{last} ($a$) Given a line bundle $L$, we denote $\Bs(L)$ its base locus. \\
($b$) We denote $U_0$  the complement in $\Pico A$ of the closed subset $V_a^1(\omega_X)$. Since, by Corollary \ref{inclusions}, \ $V^1_a(\omega_X)\supseteq\cdots\supseteq V^d_a(\omega_X)$,  it follows that, for all $\alpha\in U_0$, $h^0(\omega_X\otimes a^*\alpha)$ takes the minimal value, i.e. $\chi(\omega_X)$.\\
($c$) Given a point $p\in X$, we denote ${\mathcal B}_a(p)$ the subset (closed in $U_0$)
$${\mathcal B}_a(p)=\{\alpha\in U_0\>|\>p \in \Bs( \omega_X\otimes a^*\alpha)\}$$
($d$) We will say that a line bundle ``is birational'' to mean that the associated rational map to projective space is  birational.
\end{term/not}

 The statements we are aiming at are
\begin{thm}\label{birprov} Let $X$ be a variety of general type satisfying Hypothesis \ref{hyp}.
If, for general $p$ in $X$, $\codim_{\Pico A}{\mathcal B}_a(p)\ge 2$ then $\omega_X^2\otimes a^*\alpha$ is birational for all $\alpha\in \Pico A$. In particular, $\omega_X^2$ is birational.
\end{thm}
 \begin{cor}\label{birdef} Let $X$ be a a variety of general type satisfying Hypothesis \ref{hyp}, and assume that there exists a $\alpha\in\Pico A$ such that $\omega_X^2\otimes a^*\alpha$ is not birational. Then, for general $\beta\in U_0$, $\Bs(\omega_X\otimes a^*\beta)$ has codimension one. Moreover $X$ is covered by the divisorial components of $\Bs(\omega_X\otimes a^*\beta)$, for $\beta$ varying in $U_0$.
\end{cor}
Note that it makes sense to speak of the base locus of $\omega_X\otimes a^*\alpha$ since, by Proposition \ref{lazarsfeld}, the hypotheses of both the Theorem and the Corollary imply that $\chi(\omega_X)>0$, whence
$h^0(\omega_X\otimes a^*\alpha)>0$ for all $\alpha\in \Pico A$.
\begin{proof}[Proof of Corollary \ref{birdef}]
Let ${\mathcal B}_a$ the closed subvariety of $X\times U_0$ defined as
$${\mathcal B}_a=\{(p,\alpha)\in X\times U_0\>|\> \hbox{ $p$ is a base point of $\omega_X\otimes a^*\alpha$ } \}.$$
If, for some $\alpha\in \Pico X$, $\omega_X^2\otimes  a^*\alpha$ is not birational then, by Theorem \ref{birprov}, for general $p\in X$ we have that $\codim{\mathcal B}_a(p)=1$.
Since ${\mathcal B}_a(p)$ is the fiber of the projection $p:{\mathcal B}_a\rightarrow X$
it follows that $\codim_{X\times U_0} {\mathcal B}_a=1$. This implies that the fibers of the other projection $q: {\mathcal B}_a\rightarrow \Pico A$ have codimension 1 in $X$.\end{proof}

\begin{proof}[Proof of Theorem \ref{birprov}] We first recall some basic facts about the Fourier-Mukai transform of the sheaves $\I_p\otimes\omega_X$.
In the first place, if $p$ does not belong to $exc(a)$ then
\begin{equation}\label{GR}
R^ia_*(\I_p\otimes\omega_X\otimes a^*\alpha)=0\qquad\hbox{for $ i>0$}
\end{equation}
This follows immediately from the exact sequence
 \begin{equation}\label{standard}
0\rightarrow \I_p\otimes\omega_X\rightarrow \omega_X\rightarrow \OO_p \otimes \omega_X\rightarrow 0
\end{equation}
and the Grauert-Riemenschneider vanishing theorem. Hence, as for the canonical bundle (see the proof of Prop. \ref{trace}), the Leray spectral sequence yields that
\begin{equation}\label{GRp}V^i_a(\I_p\otimes \omega_X)=V^i(a_*(\I_p\otimes\omega_X))
\end{equation}
By sequence (\ref{standard}), tensored by $a^*\alpha$, it follows that
\begin{equation}\label{equality}V^i_a(\I_p\otimes \omega_X)=V^i_a(\omega_X)\qquad\hbox{for all }  i\ge 2
\end{equation}
Concerning the case $i=1$ we have  the surjection
$$H^1(\I_p\otimes \omega_X\otimes a^*\alpha)\twoheadrightarrow H^1(\omega_X\otimes a^*\alpha),$$
 which is an isomorphism if and only if
 $p$ is not a base point of $\omega_X\otimes \alpha$. In other words
 $$V^1_a(\I_p\otimes\omega_X)={\mathcal B}_a(p)\cup V^1_a(\omega_X)$$
 Therefore the hypothesis about ${\mathcal B}_a(p)$ ensures that
 \begin{equation}\label{codim} \codim V^1_a(\I_p\otimes\omega_X)\ge 2
 \end{equation}
  Now we distinguish two cases:

 \noindent ($a$) $\dim X<\dim A$. In this case, by Hypothesis \ref{hyp}, together with (\ref{GRp}),   (\ref{equality}) and (\ref{codim}),  $gv(a_*(\I_p\otimes\omega_X))\ge 1$. Hence, by Theorem \ref{regularity}(a), $a_*(\I_p\otimes\omega_X)$ is CGG. Therefore $\I_p\otimes \omega_X$ itself is $CGG$ outside $exc(a)$ (with respect to $a$). Since the same is true for $\omega_X$ (Corollary \ref{omega-acgg}(a)), it follows from Proposition \ref{CGG-GG}(a) that, for $p$ outside $exc (a)$ and for all $\alpha\in\Pico A$, $\I_p\otimes \omega_X^2\otimes a^*\alpha$ is globally generated outside $exc(a)$. This means that the projective map associated to $\omega_X^2\otimes a^*\alpha$ is birational (in fact an isomorphism on the complement of $exc(a)$).

 \noindent  ($b$) $\dim X=\dim A$, i.e. the map $a$ is surjective.
 Again by Hypotheses \ref{hyp}, (\ref{equality}) and (\ref{codim}), the sheaf $\I_p\otimes\omega_X$ satisfies,
 $$\codim V^i(\I_p\otimes\omega_X)\ge i+1\quad\hbox{for all $i$ such that $0<i<d$}$$
 while $$V^d(\I_p\otimes \omega_X)=\{\hat 0\} \quad\hbox{and }\quad R^d\Phi_{P_a}(\I_p\otimes\omega_X)=R^d\Phi_{P_a}(\omega_X)=\CC(\hat 0).$$
 Therefore the exact same arguments of Proposition \ref{trace} apply, proving that the torsion of $\widehat{R\,{\mathcal Hom}(a_*(\I_p),\OO_A)}$ is $\CC(\hat 0)$. Hence, by Theorem \ref{regularity}(b), $a_*(\I_p\otimes\omega_X)$ is ECGG, with $\{\hat 0\}$ as underlying subset of $\supp \tau(\widehat{R\Delta(\I_p\otimes\omega_X)})$. It follows that, for $p$ not belonging to $exc(a)$, the sheaf $\I_p\otimes\omega_X$ is $ECGG$ away of $exc(a)$. Let $W$ be the non-empty open set  of points $p\in X$ such that $\omega_X$ is $CGG$ at $p$. In view of Remark \ref{useful}, $W$ is the complement of the intersection of all base loci $\Bs(\omega_X\otimes a^*\alpha^{-1})$, for $\alpha\in \Pico A$ such that $h^0(\omega_X\otimes a^*\alpha^{-1})$ is minimal, i.e. equal to $\chi(\omega_X)$. Now let $\alpha \in \Pico A$. It follows from Proposition \ref{CGG-GG}(b) that, if $q$ is not a base point of $\omega_X\otimes a^*\alpha$ (and does not lie in $exc(a)$), then $\I_p\otimes\omega_X^2\otimes a^*\alpha$ is globally generated at $q$.
  Denoting $U_\alpha$ the complement of $exc(a)\cup {\rm Bs}(\omega_X\otimes a^*\alpha)$, we conclude that  for all $p\in U_\alpha\cap W$ the sheaf $\I_p\otimes \omega_X^2\otimes a^*\alpha$ is globally generated at all points of $U_\alpha\cap W$. As above, this means that the projective map associated to $\omega_X^2\otimes a^*\alpha$ is an isomorphism on $U_\alpha\cap W$.
\end{proof}

\section{Proof of the Theorem}
Here we prove Theorem \ref{A}  of the Introduction.

\noindent  $(b)\Rightarrow (a)$ Let $(A,\Theta)$ be an indecomposable p.p.a.v., and let $X\rightarrow \Theta$ be a desingularization. Without loss of generality, we can assume that $\Theta$ is \emph{symmetric}, i.e. $\Theta=(-1)^*\Theta$.  The restriction map $H^0(A,\OO_A(2\Theta))\rightarrow H^0(\Theta, \OO_\Theta(2\Theta))=H^0(\Theta,\omega_X^2)$ is surjective. Hence the projective map associated to $\omega_X^2$ has degree two, since $x$ and $-x$ have the same image. By a result of Ein-Lazarsfeld, $\Theta$ is normal and has rational singularities \cite[Thm. 1]{el}. Hence $\Theta$ has canonical singularities by a result of Koll\'ar \cite[Thm. 11.1(1)]{singpairs}, and therefore $H^0(X,\omega_X^2)\cong H^0(\Theta,\omega_\Theta^2)$. It follows that the bicanonical map of $X$ has degree $2$.

 The rest of the section is devoted to the proof of the  implication $(a)\Rightarrow (b)$. The argument is composed of various steps.

\subsection{Decomposition}
The hypotheses of Theorem \ref{A} are more than enough to imply that
 $alb: X\rightarrow \Alb X$, the Albanese map of $X$, satisfy  Hypothesis \ref{hyp}. Keeping the Terminology/Notation \ref{last}, it follows that, by Corollary \ref{birdef},  $|\omega_X\otimes\alpha|$ has a base divisor for all $\alpha\in U_0$. More precisely, as in the Proof of Corollary \ref{birdef}, let us consider the relative base locus
 $${\mathcal B}=\{(p,\alpha)\in X\otimes U_0\>|\> \hbox{ $p$ is a base point of $\omega_X\otimes \alpha$ } \}.$$
 equipped with the projections on the two factors, $p$ and $q$. ${\mathcal B}$ has a natural subscheme structure given by the image of the relative evaluation map $p^*(p_*{\mathcal L})\otimes{\mathcal L}^{-1}\rightarrow \OO_{X\times U_0}$, where ${\mathcal L}=p^*\omega_X\otimes P_{|X\times U_0}$. We know from Theorem \ref{birprov} that ${\mathcal B}$ has codimension $1$ and that its divisorial part is dominant on $X$ and surjects on $U_0$ via $p$ and $q$.   Let
 $${\mathcal Y}\subseteq {\mathcal B}$$
 be the union of the divisorial components of ${\mathcal B}$. Finally, for $\alpha\in U_0$, let $F_\alpha$ be the (scheme-theoretic) fiber of $q:{\mathcal Y}\rightarrow U_0$. Therefore, at a \emph{general} point  $\alpha\in U_0$, the are no other fixed divisors, so that $F_\alpha$ is the fixed divisor of $\omega_X\otimes \alpha$:  $$|\omega_X\otimes \alpha| =
 |M_\alpha|+F_\alpha$$ where $|M_\alpha|$ is the (possibly empty) mobile part. We consider the Abel-Jacobi map
 $$f_{\alpha_0}: \Pico A\rightarrow \Pico X\qquad \alpha\mapsto \OO_X(F_\alpha-F_{\alpha_0})$$
 where $\alpha_0$ is fixed in $U_0$. As it is a map between abelian varieties, it is  a morphism. By rigidity
 $$f:= f_{\alpha_0}-f_{\alpha_0}(0):\Pico A\rightarrow \Pico X$$ is a homomorphism. Note that $f$ does not depend on $\alpha_0$ since, given another suitably general $\alpha_1\in \Pico A$, $f_{\alpha_1}-f_{\alpha_0}$ is a translation.
 We have the following strong constraint on the Albanese and Picard variety of $X$:
 \begin{lem}\label{decomp} $f^2=f$ and $\Pico X$ decomposes as \ \
 $\Pico X\cong \ker f\times {\rm ker}({\rm id}- f).$
\end{lem}
 \begin{proof} Let $\alpha,\beta\in \Pico A$. We have that
 \begin{equation}\label{F}
 \OO_X(M_\alpha)\otimes\OO_X(F_\beta)=\omega_X\otimes\alpha\otimes f(\beta\otimes \alpha^{-1})
 \end{equation}
 This follows by definition of $f$ since the left hand side is isomorphic to
 $$\OO_X(M_\alpha+F_\alpha)\otimes\OO_X(F_\beta-F_{\alpha})\cong \omega_X\otimes\alpha\otimes f(\beta\otimes \alpha^{-1}).$$ Assuming that $\alpha$ and $\beta$ are \emph{general}, so that the fixed divisor of $\omega_X\otimes\alpha\otimes f(\beta\otimes \alpha^{-1})$ is $F_{\alpha\otimes f(\beta\otimes \alpha^{-1})}$ (i.e. $|M_{\alpha\otimes f(\beta\otimes \alpha^{-1})}|$ has no divisors), it follows from (\ref{F}) that
 $$F_{\beta}=F_{\alpha\otimes f(\beta\otimes \alpha^{-1})},$$
  i.e. $f({\beta})=f({\alpha\otimes f(\beta\otimes \alpha^{-1}})).$ This means that
  $ f( f(\beta\otimes \alpha^{-1}))=f(\beta\otimes \alpha^{-1})$ for general $\alpha$ and $\beta$, hence $f^2=f$. This gives the splitting of the exact sequence $\hat 0\rightarrow \ker f\rightarrow \Pico A\rightarrow {\rm Im} f\rightarrow \hat 0$  and the identification ${\rm Im} f\cong \ker({\rm id}-f)$.
 \end{proof}

 \subsection[Triviality of $\ker f$.]{Triviality of ${\bf ker \, \emph{f}}$.}
 \begin{lem}\label{zero} The morphism $f$ is injective. i.e. $\Pico X=\ker({\rm id}-f)$.
 \end{lem}
 \begin{proof} Assume that $\dim\ker f>0$. Since the complement of $U_0$ is $V^1(\omega_X)$, which is a finite set, $U_0\cap \ker f$ is a non-empty open subset of $\ker f$. By construction, for all $\beta\in U_0\cap \ker f$ we have the decomposition
 \begin{equation}\label{kerf}
 \omega_X\otimes\beta=\OO_X(M_\beta+F)
 \end{equation}
 where $F:=F_\beta$ is constant for $\beta$ varying in $U_0\cap\ker f$.  We make the following\\ Claim: \emph{for \emph{general} $\beta\in U_0\cap ker f$, $F$ is the fixed divisor of $\omega_X\otimes \beta$ \emph(in other words: $\OO_X(M_\beta)$ has no fixed divisors)}.\\ Admitting the Claim for the moment, we conclude the proof.   By
  Lemma \ref{decomp}, \ $\Alb X\cong B\times C$, where $B$ and $C$ are the dual abelian varieties of $\ker f$ and $\ker({\rm id}-f)$. Since $B\ne 0$, composing the Albanese map of $X$ with the projection on $B$ we get a morphism $b:X\rightarrow B$ satisfying      Hypothesis \ref{hyp}. Indeed $V^i_b(\omega_X)=V^i(\omega_X)\cap \ker f$ and therefore the hypothesis of Theorem \ref{A} yields  that $\dim V^i_b(\omega_X)=0$ for all $i>0$. Moreover, $\Pico B=\ker f$ is embedded in $\Pico X$. By Proposition \ref{finiteness}, $b$ is generically finite. By the Claim, the divisorial part of the base locus of $\omega_X\otimes b^*\beta$ is constant on a non-empty Zariski open set of $\Pico B$. But, since $\omega_X^2$ is not birational, this is in contrast with Corollary \ref{birdef} (last part of the statement) applied to $b$. Hence $\ker f$ has to be trivial. \smallskip\\
  To prove the Claim, we notice that, for $\gamma\in \ker({\rm id}-f)\cap U_0$, we have the decomposition $\omega_X\otimes\gamma=\OO_X(M+F_{\gamma})$, where $F_{\gamma}$ is  fixed and $\OO_X(M):=\OO_X(M_\gamma)$ is constant on $U_0\cap \ker({\rm id}-f)$. From this, (\ref{kerf}) and (\ref{F}) it follows that $|\omega_X\otimes\beta\otimes \gamma|=|M_\beta|+F_\gamma$, with $\beta\in\ker f\cap U_0$ and $\gamma\in\ker ({\rm id}-f)\cap U_0$. In other words $M_{\beta\otimes\gamma}=M_\beta$ and $F_{\beta\otimes\gamma}=F_\gamma$. For $\beta$ and $\gamma$ general, $\beta\otimes \gamma$ is general in $\Pico A$. The Claim follows since  we know that $|M_{\beta\otimes\gamma}|$ has no  fixed divisors.
  \end{proof}
  \subsection[The Poincar\'e line bundle on $X\times \Pico X$. ]{The Poincar\'e line bundle on ${\bf \emph{X}\times Pic^0\, \emph{X}}$. }
   It follows from Lemma \ref{zero} that, for all $\alpha\in U_0$,
  \begin{equation}\label{omega}
  |\omega_X\otimes\alpha|=|M|+F_\alpha,
  \end{equation}
   where $F_\alpha$ (the fiber of the projection ${\mathcal Y}\rightarrow \Pico X$) is the fixed divisor, hence $\chi(\omega_X)=h^0(M)$.

  Let  $\overline{{\mathcal Y}}$  be the closure of ${\mathcal Y}$ in $X\times \Pico X$. Moreover, for $p\in X$ let ${\mathcal D}_p$ be the fiber of the projection $\overline{\mathcal Y}\rightarrow X$.  For general $p\in X$, we have that
  ${\mathcal D}_p$ is the closure of the union of the divisorial components of the locus of $\alpha\in U_0$ such that $p\in \Bs(\omega\otimes\alpha)$. Since the Albanese map is defined up to a translation in $\Alb X$, we can assume that there is one such point, say $\bar p$, such that $alb(\bar p)=0$ in $\Alb X$. Finally, we recall from \S1 that $P$ denotes the Poincar\'e line bundle on $X\times\Pico X$.
  \begin{lem}\label{poincare'} \ \ \ \ \ \ \ \ $P\cong \OO_{X\times\Pico X}(\overline{\mathcal Y})\otimes p^*(\omega_X^{-1}\otimes M)\otimes q^*\OO_{\Pico X}(-{\mathcal D}_{\bar p})$
  \end{lem}
  \begin{proof} By the definition of $\overline{\mathcal Y}$ and (\ref{omega}) we have that
  $$ \OO_{X\times\Pico X}(\overline{\mathcal Y})\otimes p^*(\omega_X^{-1}\otimes M)\otimes q^*\OO_{\Pico X}(-{\mathcal D}_{\bar p}),$$ - restricted to $X\times\{\alpha\}$ is isomorphic to $\OO_X(F_\alpha)\otimes \OO_X(F_\alpha)^{-1}\otimes\alpha=\alpha=P_{|X\times \{\alpha\}}$, for all $\alpha\in U_0$;\\- restricted to $\{\bar p\}\times\Pico X$ is isomorphic to
  $\OO_X({\mathcal D}_p)\otimes\OO_X(-{\mathcal D}_p)$, i.e. trivial. \\
  The Lemma follows from the see-saw principle.
  \end{proof}
  \subsection{Conclusion of the proof of Theorem \ref{A}}
  \begin{lem}\label{conclusion}  $M=\OO_X$, i.e. $
 \chi(\omega_X)=1$
\end{lem}
Having proved this, the implication $(a)\Rightarrow (b)$ of Theorem \ref{A} follows from Proposition \ref{theta1}.
\begin{proof}[Proof of Lemma \ref{conclusion}] Lemma \ref{poincare'} yields the standard short  exact sequence on $X\times \Pico X$
$$
0\rightarrow P^{-1}\buildrel{\cdot \overline{\mathcal Y}}\over\rightarrow p^*(\omega_X\otimes M^{-1})\otimes q^*(\OO_{\Pico X}({\mathcal D}_{\bar p}))\rightarrow p^*(\omega_X\otimes M^{-1})\otimes q^*(\OO_{\Pico X}({\mathcal D}_{\bar p}))_{|\overline{\mathcal Y}}\rightarrow 0
$$
Applying $ R^dq_*$, and recalling that $R\Phi_{P^{-1}}\cong (-1)_{\Pico X}^*R\Phi_P$ (see 2.1) we get
\begin{equation}\label{short} 0\rightarrow (-1)^*_{\Pico X}\widehat{\OO_X}\buildrel\mu\over\rightarrow \OO_{\Pico X}({\mathcal D}_{\bar p})^{\oplus \chi(\omega_X)}\rightarrow \tau\rightarrow 0\end{equation}
where:\\
($a$) The rank $\chi(\omega_X)$ in the middle appears because $h^d(\omega_X\otimes M^{-1})=h^0(M)\buildrel{(\ref{omega})}\over=\chi(\omega_X)$;\\
($b$) $\tau=R^dq_*p^*(\omega_X\otimes M^{-1})\otimes q^*(\OO_{\Pico X}({\mathcal D}_{\bar p}))_{|\overline{\mathcal Y}}$ is supported at the locus of the  $\alpha\in\Pico X$ such that the fiber of the projection $q:\overline{\mathcal Y}\rightarrow \Pico X$ has dimension $d$, i.e. it coincides with $X$. Such locus is contained in the union of the $V^i(\omega_X)$ for $i>0$, hence it is finite set;\\
($c$) the map $\mu$ is injective since it is a generically surjective map of sheaves of the same rank (recall that ${\rm rk\,}\widehat{\OO_X}=\chi(\omega_X)$), and, as $gv(\omega_X)\ge 1$, the source $\widehat{\OO_X}$ is torsion free (Theorem \ref{GV1}).\\
($d$) Since $\mu$ is $R^dq_*(m_s)$, where $m_s$ is the multiplication for the section defining $\overline{\mathcal Y}$, $\mu$ is zero on the locus where the fiber of $q$ has dimension $d$. Therefore $\tau$ is a (trivial) sheaf of rank  equal to $\chi(\omega_X)$, supported at a finite set of points contained in the finite set $\bigcup_{i>0}V^i(\omega_X)=V^1(\omega_X)$.

From (\ref{short}) and the fact that ${\rm supp}\,\tau$ is a finite scheme it follows that  ${\mathcal E}xt^i(\widehat{\OO_X},\OO_{\Pico X})\cong  {\mathcal E}xt^{i+1}(\tau,\OO_{\Pico X})=0$ if $i\ne q(X)-1$. On the other hand, by Corollary \ref{duality} and Proposition \ref{appendix} of the Appendix,  ${\mathcal E}xt^d(\widehat{\OO_X},\OO_{\Pico X})\cong \CC(\hat 0)$. It follows that $d=q(X)-1$ and that $\tau=\CC(\hat 0)$. The assertion follows since the length of $\tau$ is equal to $\chi(\omega_X)$.
\end{proof}

\section{Appendix: a useful lemma on the  generalized Fourier-Mukai transform of the canonical sheaf}
As usual, let $a:X\rightarrow A$ be a morphism from a $d$-dimensional smooth projective variety to an abelian variety (over any algebraically closed field $k$). Let $P_a=(a\times{\rm id})^*{\mathcal P}$, where ${\mathcal P}$ is the Poincar\'e line bundle on $A\times \Pico A$. Assume furthermore that the induced morphism $a^*:\Pico A\rightarrow \Pico X$ is an \emph{embedding}. Then the top cohomological support locus is $V^d(\omega_X)=\{\hat 0\}$. By base change \cite[Cor. 3, p. 53]{ab}, it follows that, for $\alpha\in \Pico A$,
 \begin{equation}\label{fiber}R^d\Phi_{P_a}(\omega_X)\otimes k(\hat \alpha)\cong \begin{cases}k(\hat 0)&\hbox{if } \alpha=\hat 0\\ 0&\hbox{otherwise}
 \end{cases}
 \end{equation}
  It follows that $R^d\Phi_{P_a}(\omega_X)$ is a  sheaf supported at $\hat 0$. The result we are aiming at is
\begin{prop}\label{appendix} $R^d\Phi_{P_a}(\omega_X)\cong k(\hat 0) $ \footnote{When $X$ itself is an abelian variety (or a complex torus) this statement is well known  (see \cite[Cor. 14.1.6]{lb} and \cite{kempf} for an elementary proof in the complex case and \cite[p. 202]{huy}, \cite[p. 128]{ab} for arbitrary characteristic). This fact is crucial in the proof of Mukai Inversion Theorem \ref{mukai} }
\end{prop}
\begin{proof}
Let $B= \OO_{\Pico A,\hat 0}$ and $\mathfrak{m}$ its maximal ideal. By Nakayama's lemma, (\ref{fiber}) implies that $R^d\Phi_{P_a}(\omega_X)$ is supported only at $\hat 0$ and that $R^d\Phi_{P_a}(\omega_X)\cong B/J$, where $J$ is a $\mathfrak{m}$-primary ideal.

\noindent{\bf Claim. } \emph{ ${P_a}_{|X\times Spec \, B/J}$ is trivial. }\\
\emph{Proof of the Claim. } Let $\E$ be a sheaf on $\Pico A$. By Grothendieck duality
\begin{equation}\label{gd}
R{\mathcal H}om_{\Pico A}(R\Phi_{P_a}(\omega_X),\E)\cong  Rq_*({\mathcal H}om_{X\times \Pico A}(P_a,q^*\E))[d]
\end{equation}
Indeed,
\begin{eqnarray*}
R{\mathcal H}om_{\Pico A}(R\Phi_{P_a}(\omega_X),\E)&= & R{\mathcal H}om_{\Pico A}(Rq_*(p^*\omega_X\otimes P_a),\E)\\
&\buildrel{GD}\over\cong & R {q}_* ( R\mathcal{H}om_{X\times \Pico A}(p^*
\omega_X \otimes P_a,p^*\omega_X\otimes q^*\E[d]))\\
&\cong & Rq_*({\mathcal H}om_{X\times \Pico A}(P_a,q^*\E))[d]
\end{eqnarray*}
Therefore we have a fourth quadrant spectral sequence
$$E_2^{i,-j}={\mathcal E}xt^i_{\Pico A}(R^j\Phi_{P_a}(\omega_X),\E)\Rightarrow R^{i-j+d}q_*({\mathcal H}om_{X\times \Pico A}(P_a,q^*\E))$$
Clearly the term $E_2^{i,-j}$ is non-zero only if $i\ge 0$. Assuming $i\ge 0$, in the case $i-j+d=0$, i.e. $j=i+d$ we have that $R^j\Phi_{P_a}(\omega_X)$ is non zero if and only if $j=d$, i.e. $i=0$. In conclusion for $i-j+d=0$ the only non-zero $E_2$-term is $E_2^{0,-d}={\mathcal H}om_{\Pico A}(R^d\Phi_{P_a}(\omega_X),\E)$. Since the differentials from and to $E^{0,-d}_2$ are zero, we get that
$${\mathcal H}om_{\Pico A}(R^d\Phi_{P_a}(\omega_X),\E)=E_2^{0,-d}=E^{0,-d}_\infty\cong q_*({\mathcal H}om_{X\times \Pico A}(P_a,q^*\E)).$$
Taking global sections, we get the isomorphism (functorial in $\E$)
\begin{equation}\label{HOM}
{\rm Hom}_{\Pico A}(R^d\Phi_{P_a}(\omega_X),\E)\cong {\rm Hom}_{X\times\Pico A}(P_a,q^*\E)
\end{equation}
Using the previous isomorphism twice, once for $\E=B/J$ and the other for $\E=k(\hat 0)$, by functoriality we get the commutative diagram
\begin{equation}\label{diagram}
\xymatrix{B/J \ar@{=}[r]\ar[d]&{\rm Hom}(B/J,B/J)\ar@{=}[r]\ar[d]&{\rm Hom}({P_a}_{|X\times Spec\,B/J},\OO_{X\times Spec\,B/J})\ar[d]
\\
k(\hat 0)\ar@{=}[r]&{\rm Hom}(B/J,k(\hat 0))\ar@{=}[r]&{\rm Hom}({P_a}_{|X\times \{\hat 0\}},\OO_{X\times \{\hat 0\}})}
\end{equation}
Since ${P_a}_{|X\times\{\hat 0\}}$ is trivial, we can take an isomorphism $h\in {\rm Hom}({P_a}_{|X\times \{\hat 0\}},\OO_{X\times \{\hat 0\}})$. By the diagram above, $h$ lifts to a morphism $\bar h: {P_a}_{|X\times Spec\,B/J}\rightarrow \OO_{X\times Spec\,B/J}$. Since $\bar h$ is a map between invertible sheaves on $X\times Spec\,B/J$ which is an isomorphism when restricted to $X\times\{\hat 0\}$, $\bar h$ is an isomorphism. Therefore ${P_a}_{|X\times Spec\,B/J}$ is trivial. The Claim is proved.

At this point the Proposition follows since   \emph{the smooth point $\hat 0$ of $\Pico A$ is the maximal subscheme $Z$ of $\Pico A$ such that ${P_a}_{|X\times Z}$ is trivial} (\cite{ab} \S10). This in turn follows from the well known fact (\cite{ab} \S13) that {the smooth point $\hat 0$ of $\Pico A$ is the maximal subscheme $Z$ of $\Pico A$ such that ${\mathcal P}_{|A\times Z}$ is trivial}, combined with
 the fact that $a^*\Pico A\rightarrow \Pico X$ is an embedding.
However,  we provide  an equivalent but self-contained argument. Let us consider, in analogy to Subsection 2.3, the functor $R\Psi_a:{\bf D}(\Pico A)\rightarrow {\bf D}(X),$ defined by $R\Psi_a(\cdot)=Rp_*(P_a\otimes q^*(\cdot))$. Since $P_a=(a\times{\rm id}_{\Pico A})^*{\mathcal P}$, it follows that
 \begin{equation}\label{pullback}
 R\Psi_a\cong La^*\circ R\Psi
 \end{equation}
 The Claim   implies, by the K\"unneth formula, that
 \begin{equation}\label{banale}
 R\Psi_{P_a}(B/J)=R^0\Psi_{P_a}(B/J)=\OO_X^{\oplus r}
 \end{equation}
  where $r=length\,B/J$. On the other hand, by \cite[Lemma 4.8]{semihom},
 $R\Psi_{\mathcal P}(B/J)=R^0\Psi_{\mathcal P}(B/J):=U$, where $U$ is a \emph{unipotent} vector bundle on $A$ of rank $r$, i.e. a vector bundle having a filtration  $0=U_0\subset U_1\subset\cdots\subset U_{r-1}\subset U_r=U$, such that $U_i/U_{i-1}\cong\OO_A$.
 By (\ref{pullback}) and (\ref{banale}) it follows that $a^*U$ is trivial. The filtration of $U$, pulled back via $a$, induces the  filtration of the trivial bundle:
 $$0\subset a^*U_1\subset\cdots\subset a^*U_{r-1}\subset a^* U_r=\OO_X^{\oplus r},$$
where $a^*U_i/a^*U_{i-1}\cong \OO_X$.  Since $h^0(X,a^*U_i)\le i$ for all $i$, the fact that $a^*U_r$ is trivial implies easily, by descending induction on $i$, that
\begin{equation}\label{splitting}
h^0(X, a^*U_i)=i\quad\hbox{for all } i
\end{equation}
This in turn implies that
the sequence
$$0\rightarrow a^*U_{i-1}\rightarrow a^*U_i\rightarrow\OO_X\rightarrow 0$$
splits  for all $i$ (the coboundary map $H^0(\OO_A)\rightarrow H^1(U_{i-1})$ is zero).
In particular the extension
$$0\rightarrow \OO_X\rightarrow a^*U_2\rightarrow\OO_X\rightarrow 0$$
is split. But the natural pullback map \begin{equation}\label{diff}
H^1(\OO_A)\cong {\rm Ext}^1(\OO_A,\OO_A)\rightarrow {\rm Ext}^1(\OO_X,\OO_X)\cong H^1(\OO_X)
\end{equation} is identified with the differential at $\hat 0$ of the map $a^*:\Pico A\rightarrow \Pico X$.   Since $a^*$ is assumed to be an embedding, (\ref{diff}) is injective. Hence also the extension
$$0\rightarrow \OO_A\rightarrow U_2\rightarrow \OO_A\rightarrow 0$$
is split. This yields that $h^n(U)\ge 2$. But this is impossible since, by Mukai's inversion, $R\Phi_{\mathcal P}(U)=\hat{U}[-n]\cong(-1_{\Pico A})^*B/J[-n]$, and therefore, by base change, $H^n(U)\cong (B/J)\otimes k(\hat 0)\cong k(\hat 0)$. In conclusion $r=length(B/J)=1$, i.e. \ $B/J\cong k(\hat 0)$.
\end{proof}

\providecommand{\bysame}{\leavevmode\hbox
to3em{\hrulefill}\thinspace}

\end{document}